\documentclass[12pt]{amsart}

\usepackage[]{graphicx}
\usepackage[abs]{overpic}
\usepackage{amssymb}
\usepackage{amsmath}
\usepackage{amsthm}
\usepackage{mathrsfs}
\usepackage{comment}

\newcommand{\ack}[1]{\par \noindent \textbf{Acknowledgements.} #1 }

\theoremstyle{plain}
\newtheorem{thm}{Theorem}[section]

\newtheorem{prop}[thm]{Proposition}

\theoremstyle{definition}
\newtheorem{defn}[thm]{Definition}
\newtheorem{exam}[thm]{Example}

\title{Linking Numbers for Handlebody-Links}
\author{Atsuhiko Mizusawa}
\address{Department of Mathematics, Faculty of Fundamental Science and Engineering, Waseda University, 3-4-1 Okubo, Shinjuku-ku, Tokyo 169-8555, Japan}
\email[Atsuhiko Mizusawa]{a\_mizusawa@suou.waseda.jp}

\begin{document}

\begin{abstract}
As a generalization of the linking number, we construct a set of invariant numbers for two-component handlebody-links. These numbers are elementary divisors associated with the natural homomorphism from the first homology group of a component to that of the complement of another component.
\end{abstract}

\maketitle

\section{Introduction} \label{sec1}
\par
\textit{A handlebody-link} is an embedding of handlebodies into a 3-manifold \cite{Ishi}. Especially an embedding of one handlebody into a 3-manifold is called \textit{a handlebody-knot}. In case the genus of each component of a handlebody-link is one, it can be regarded as an ordinary link. 
\par
Handlebody-links are also regarded as \textit{neighborhood equivalence} class of spatial graphs \cite{Su}. Two spatial graphs are neighborhood equivalent if they have isotopic regular neighborhoods. Hence handlebody-links are represented by spatial graphs whose regular neighborhoods are isotopic to the handlebody-links. In the present paper, we use this representation for handlebody-links. \textit{A contraction move} of spatial graphs is a local transformation of spatial graphs shown in Fig. \ref{fig0}; contracting an edge $e$ and its inverse move, where this move is done in an embedded disc in a 3-manifold. In \cite{Ishi}, it is shown that two spatial graphs are neighborhood equivalent (i.e. they represent the same handlebody-links) if and only if they are transformed to each other by a sequence of contraction moves.
\begin{figure}[t]
$$
\raisebox{-27 pt}{\begin{overpic}[bb=0 0 559 352, width=100 pt]{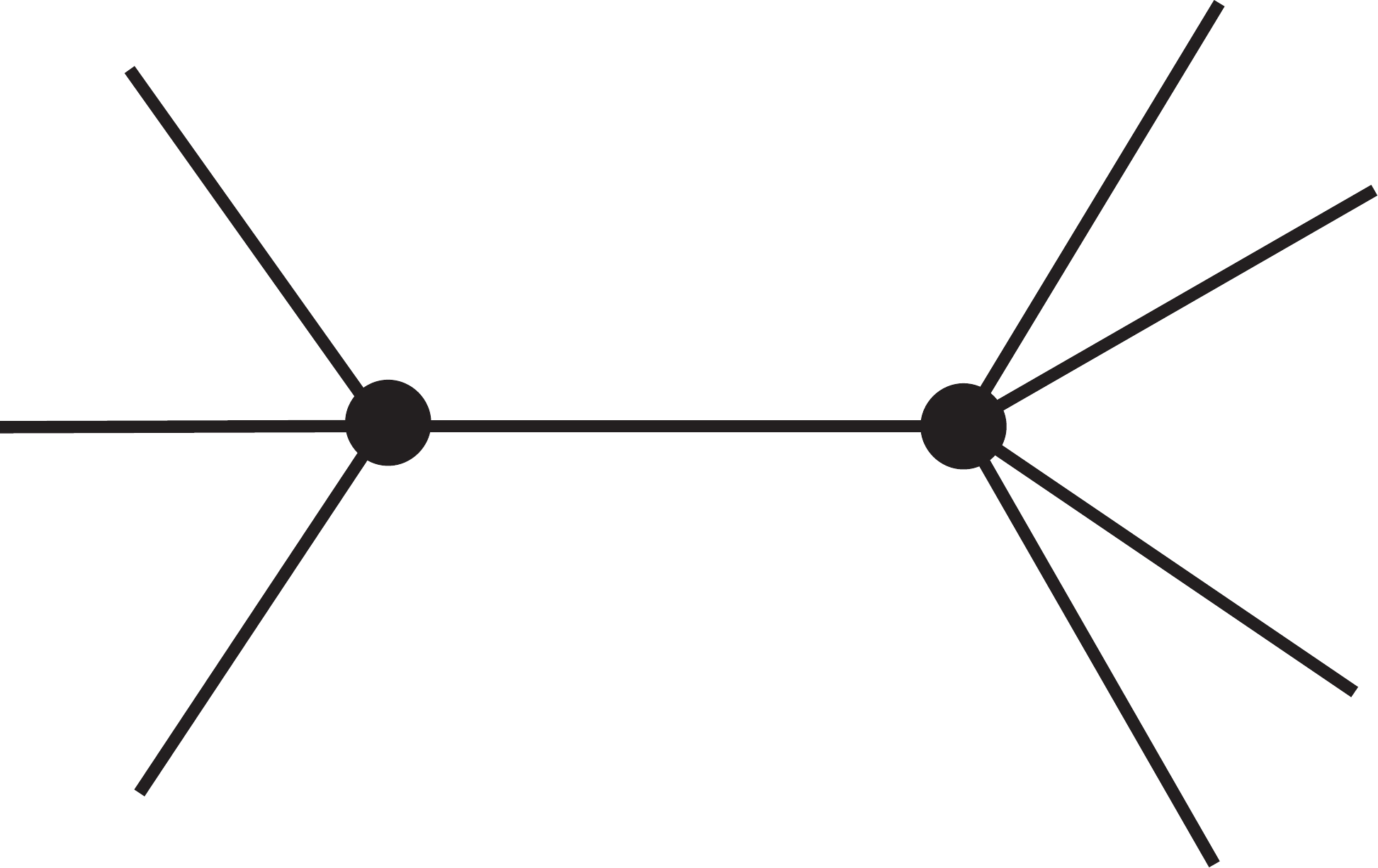}
\put(46, 23){$e$}
\end{overpic}}
\hspace{0.3 cm}\mbox{\LARGE{$\leftrightarrow$}}\hspace{0.3 cm}
\raisebox{-27 pt}{\includegraphics[bb= 0 0 323 352, height=63 pt]{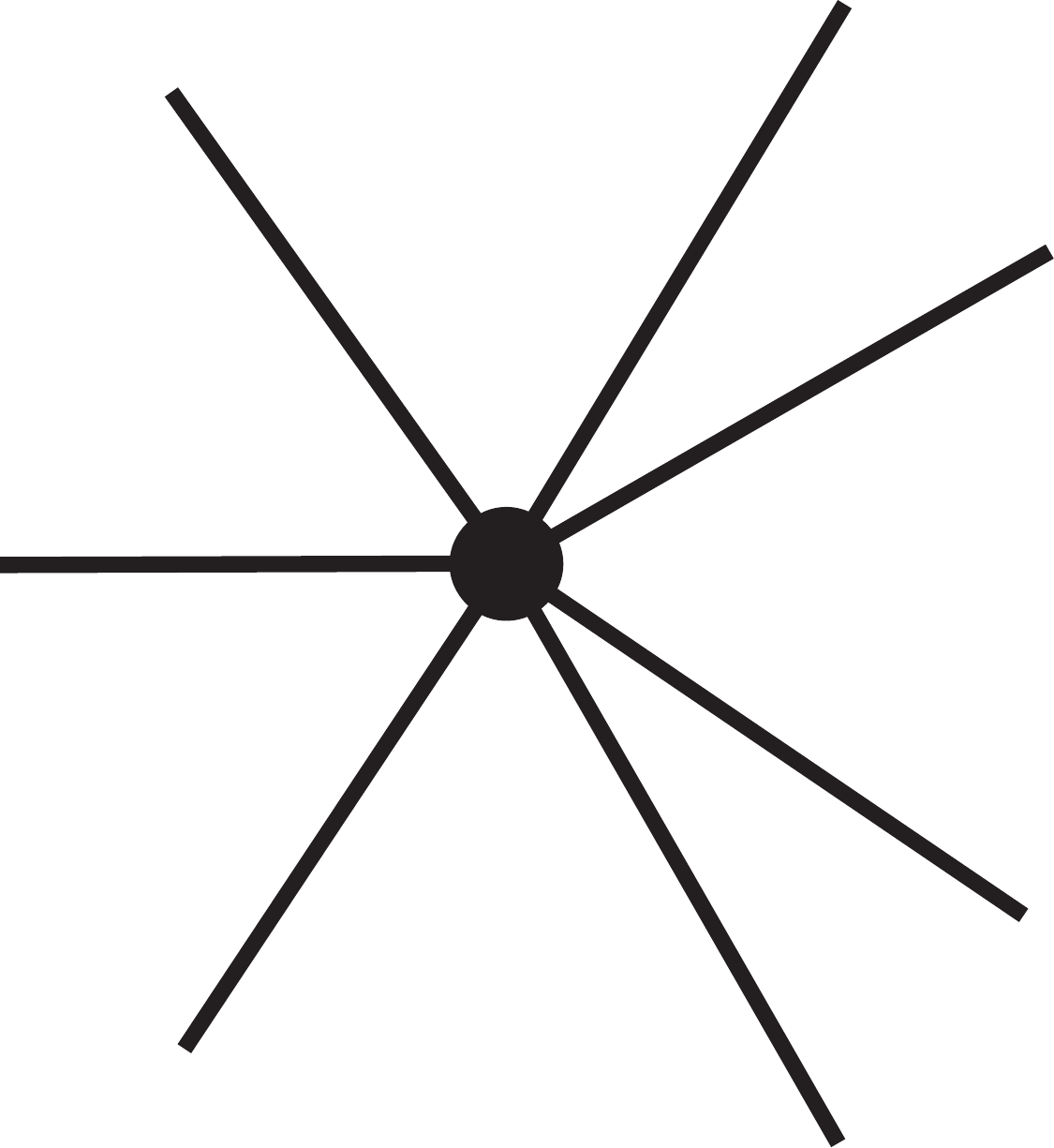}}
$$
\caption{A contraction move on the edge $e$.}\label{fig0}
\end{figure}
\par
In the ordinary knot theory, the linking number is an elementary and important invariant for oriented two-component links. In this paper, we construct an invariant for two-component handlebody-links in a 3-sphere $S^3$ like the linking number for ordinary links. We consider two first homology groups for each component of the two-component handlebody-links. The invariants for two-component handlebody-links are elementary divisors of a matrix whose entries are linking numbers of closed circles corresponding to the basis of one homology group and that of the other homology group. When a genus of each component of handlebody-links is one, thus they are ordinary links, our invariants coincide with the linking number of them. 
\par
The rest of the present paper is organized as follows. In Section \ref{sec2}, we define the invariant of two-component handlebody-links. We show an example of the invariant in Section \ref{sec3}. In Section \ref{sec4}, we see a geometrical meaning of the invariant that is a relationship to first homology groups of exterior spaces of handlebody-link components.

\section{Definition} \label{sec2}
\par
In this section, we construct the invariant numbers for two-component handlebody-links in $S^3$. In case both components are genus 1 (i.e. two-component links), these numbers correspond to the linking number.

\begin{defn}(\textit{Linking numbers for handlebody-links}) Let $L$ be a two-component handlebody-link in $S^3$, $h_1$ and $h_2$ be its components and $m, n$ be genera of them respectively. We fix bases of first homology groups $H_1(h_1)$ and $H_1(h_2)$ of $h_1$ and $h_2$; $e_1, \dots, e_m$ for $H_1(h_1)$ and $f_1, \dots , f_n$ for $H_1(h_2)$. $e_i$ and $f_j$ can be regarded as embedded closed oriented circles in the $S^3$. Then we make a $m \times n$ matrix $M$ whose $(i, j)$ entry is a linking number of $e_i$ and $f_j$:
\[
\left( \begin{array}{c c c}
lk(e_1, f_1) & \cdots & lk(e_1, f_n) \\
\vdots & \ddots & \vdots \\
lk(e_m, f_1) &\cdots &  lk(e_m, f_n)
\end{array} \right).
\]
Then we consider the elementary divisors $d_1|d_2|\dots|d_l$ (for some $0 \leq l \leq \min(m,n)$, $d_i \in \mathbb{Z}$) of $M$ as $\mathbb{Z}$-module. Note that $d_1, d_2, \dots , d_l$ are unique up to signs. If $0<l$, choosing positive signs, we define 
$$ Lk(h_1, h_2) = \{|d_1|, |d_2|, \dots , |d_l| \},$$
otherwise $Lk(h_1, h_2) = \{0\}$. 
\end{defn}

\par
If $h_1$ and $h_2$ are separated, the linking numbers $Lk(h_1, h_2)$ is equal to $\{0\}$. We remark that the linking number above $lk(\cdot, \cdot): H_1(h_1)\times H_1(h_2) \rightarrow \mathbb{Z}$ is bilinear for sums of elements of the first homology groups and scalar multiplying over $\mathbb{Z}$, i.e. $lk(cx_1+dx_2, y)=c\,lk(x_1, y)+d\,lk(x_2, y)$, for $x_1, x_2 \in H_1(h_1), y \in H_1(h_2)$ and $c, d \in \mathbb{Z}$, and the same relation holds for the second argument.

\begin{thm} 
$Lk(h_1, h_2)$ is independent of the ways taking bases of $H_1(h_1)$ and $H_1(h_2)$. Thus $Lk(h_1, h_2)$ is an invariant of two-component handlebody-links.
\end{thm}
\begin{proof}
Changing a basis of the first homology group $H_1(h_1)$ (resp. $H_1(h_2)$) causes multiplying an $m\times m$ (resp. an $n\times n$) regular matrix with integer entries from left (resp. right) to $M$, and these operations do not change the elementary divisors of $M$.
\end{proof}

\section{Example} \label{sec3}
\par
We show a calculation of $Lk(h_1, h_2)$ for the handlebody-link $L= h_1 \cup h_2$ in Fig. \ref{fig1}, where we represent components of the handlebody-link by bouquet graphs. The original handlebody-link is a regular neighborhood of the graphs. By virtue of the bouquet graph representation, we see bases of first homology groups $H_1(h_1)$ and $H_1(h_2)$ of the components explicitly as the oriented loop edges $e_1, e_2, e_3$ and $f_1, f_2, f_3, f_4$ of the graphs.
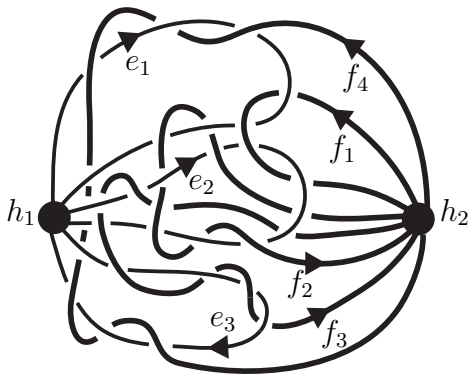
\begin{figure}[ht]
\[
\begin{overpic}[bb=0 0 421 389, width=150 pt]{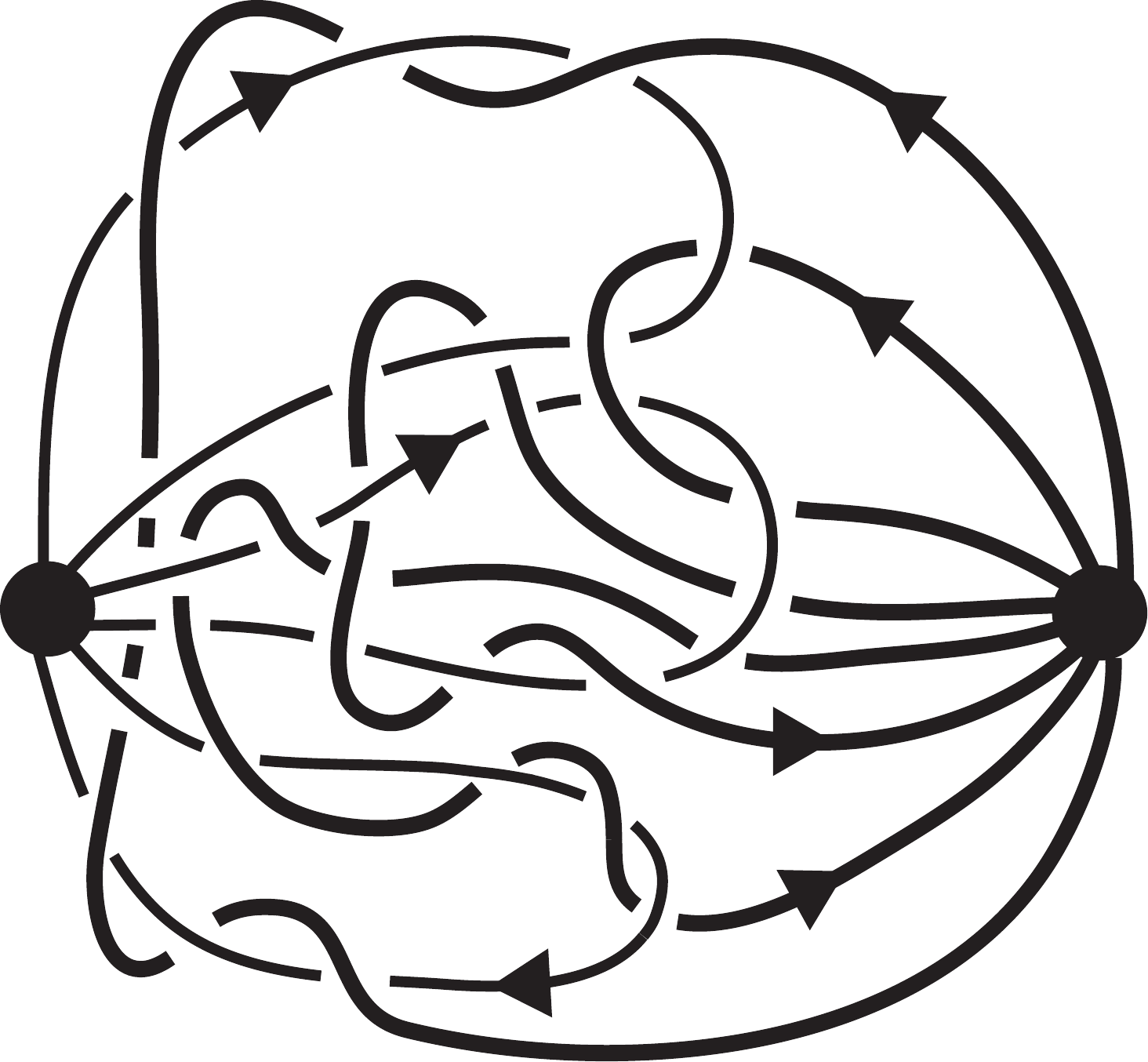}
\put(33, 115){$e_1$}
\put(57, 70){$e_2$}
\put(65, 18){$e_3$}
\put(110, 82){$f_1$}
\put(94, 31){$f_2$}
\put(107, 12){$f_3$}
\put(115, 109){$f_4$}
\put(-12, 58){$h_1$}
\put(152, 58){$h_2$}
\end{overpic}
\]
\caption{A handlebody-link $L= h_1 \cup h_2$.}\label{fig1}
\end{figure}
\begin{exam}
A matrix $M$ of linking numbers of the bases of the first homology groups of $h_1$ and $h_2$ is calculated as 
\[
\left( \begin{array}{c c c c}
-1 & -1 & 0 & 2 \\
1 & -3 & -2 & 0 \\
0 & 0 & 2 & -2 
\end{array} \right).
\]
The elementary divisors of $M$ are $1\,|\,2\,|\,4$ up to signs. Thus $Lk(h_1, h_2) =\{1,2,4\}.$ 
\end{exam}
\par
We remark that by sliding a loop along another loop, we have a new bouquet graph representing the same handlebody-link. This corresponds to adding or subtracting a row or column to other row or column. Therefore we can choose bouquet graphs such that the non-zero entries of the linking matrix of their loops is exactly the elementary divisors of the matrix. Thus the linking number of the handlebody-link is realized by such bouquet graphs.

\section{Geometrical meanings of  $Lk(h_1, h_2)$} \label{sec4}
\par
In this section, we show a relationship between the linking numbers $Lk(h_1, h_2)$ of handlebody-links and the first homology group of a complement space of one component $h_1$ (resp. $h_2$) in $S^3$ with relations derived from another component $h_2$ (resp. $h_1$). 
\par
First we consider the first homology group $H_1(S^3\setminus h_1)$ of the complement of $h_1$. Remark that $H_1(S^3\setminus h_1) = \mathbb{Z}^{m}$ where $m$ is a genus of $h_1$. Let $\{f_1, f_2, \dots, f_n\}$ be a basis of the first homology group $H_1(h_2)$ of $h_2$ where $n$ is a genus of $h_2$. Since $h_2$ is in the complement space $S^3\setminus h_1$, we can regard the elements of the basis as oriented closed circles in $S^3\setminus h_1$ and as elements of $H_1(S^3\setminus h_1)$. Then we consider a quotient group of $H_1(S^3\setminus h_1)$ by a subgroup generated by $f_1, f_2, \dots, f_n$;
$$A_1 = H_1(S^3\setminus h_1)/\left<f_1, f_2, \dots, f_n \right>.$$
$A_1$ is independent of choices of the basis of $H_1(h_2)$ and determined only by $L$. $A_1$ is a finitely generated abelian group and have a form
$$A_1 = \mathbb{Z}^{m-k} \oplus {\rm Tor}(A_1),$$
for some $k$, where ${\rm Tor}(A_1)$ is a torsion part of $A_1$. If the component $h_2$ is separated from $h_1$, the circles of the basis of $H_1(h_2)$ are trivial circles in $S^3\setminus h_1$ and $A_1= H_1(S^3\setminus h_1)= \mathbb{Z}^{m}$. By tangling $h_2$ to $h_1$, a part of $\mathbb{Z}^{m}$ degenerates to the torsion part ${\rm Tor}(A_1)$. So we think ${\rm Tor}(A_1)$ has information of a tangle between $h_1$ and $h_2$. 
\par
On the other hand, we can construct a group $A_2$ as above exchanging the roles of $h_1$ and $h_2$, 
$$A_2 = \mathbb{Z}^{n-k'} \oplus {\rm Tor}(A_2),$$
for some $k'$, where ${\rm Tor}(A_2)$ is a torsion part of $A_2$. ${\rm Tor}(A_2)$ also has information of a tangle between $h_1$ and $h_2$. Then we have the next proposition.
\begin{prop}
Suppose that $m$ is not smaller than $n$. Then $A_1 \cong \mathbb{Z}^{m-n} \oplus A_2$.
\end{prop}
\begin{proof}
It is sufficient to show ${\rm Tor}(A_1)\cong{\rm Tor}(A_2)$. We see a form of ${\rm Tor}(A_1)$. Let $\Gamma_1$ and $\Gamma_2$ be bouquet graphs representing the components $h_1$ and $h_2$ respectively, i.e. the regular neighborhood of $\Gamma_1$ (resp. $\Gamma_2$) is equivalent to $h_1$ (resp. $h_2$). Let $\{e_1, e_2, \dots, e_m \}$ and $\{f_1, f_2, \dots, f_n \}$ be edges of $\Gamma_1$ and $\Gamma_2$. They are regarded as bases of $H_1(h_1)$ and $H_1(h_2)$. We take a basis $\{x_1, x_2, \dots, x_m \}$ of $H_1(S^3\setminus h_1) = H_1(S^3\setminus \Gamma_1)=\mathbb{Z}^{m}$ such that $x_i$ is a closed oriented circle that rounds the edge $e_i$ of $\Gamma_1$ one time. Each $f_j$, as a close oriented circle in $S^3\setminus h_1$, is represented by the basis of $H_1(S^3\setminus \Gamma_1)$ as
\begin{eqnarray*}
& &f_j = a_{1j} x_1 + a_{2j} x_2 + \cdots + a_{mj} x_m, \,\,\, a_{ij} \in \mathbb{Z}, \\ 
& &\hspace{5.8cm}(1\leq j \leq n).
\end{eqnarray*}
From the choice of basis, the coefficient $a_{ij}$ is how many times the edge $f_j$ of $\Gamma_2$ rounds the edge $e_i$ of $\Gamma_1$. These observations lead a representation of $A_1 = H_1(S^3\setminus h_1)/\left<f_1, f_2, \dots, f_n \right>$;
$$\mathbb{Z}^{n} \overset{\psi}{\longrightarrow} \mathbb{Z}^{m} (= H_1(S^3\setminus h_1)) \overset{\varphi}{\longrightarrow} A_1 \longrightarrow 0,$$
which is an exact sequence of $\mathbb{Z}$-module homomorphisms such that ${\rm \,Im}\, \psi={\rm \,ker} \,\varphi$ and $\varphi$ is a surjection, where a matrix representation for $\psi$ is equal to $M$ in Section \ref{sec2} (assuming appropriate orientations for edges of $\Gamma_1$ and $\Gamma_2$). By appropriate change of basis of $H_1(S^3\setminus h_1)$ and relations for them (i.e. multiplying regular matrices from both sides of $M$), we have a rectangular diagonal matrix 
\[
\left( \begin{array}{c c c c c c}
d_1 &  & & & & \\
 & d_2 & & & O & \\
 & &  \ddots & & & \\
 & & & d_l & & \\
 & & & & 0 & \\
 & O & & & & \ddots \\
 & & & & & 
\end{array} \right),
\]
where $d_1|d_2| \cdots |d_l$ are elementary divisors of $M$. Therefore $A_1$ have a form 
$$A_1 = \mathbb{Z}^{m-l} \oplus \mathbb{Z} / d_1 \mathbb{Z} \oplus \cdots \oplus \mathbb{Z} / d_l \mathbb{Z}$$
and we have
$${\rm Tor}(A_1) = \mathbb{Z} / d_1 \mathbb{Z} \oplus \cdots \oplus \mathbb{Z} / d_l \mathbb{Z}.$$
\par 
We do the same operation for $H_1(S^3\setminus h_2)$ as above and have a representation of the abelian group $A_2$; 
$$\mathbb{Z}^{m} \overset{\psi'}{\longrightarrow} \mathbb{Z}^{n} (= H_1(S^3\setminus h_2)) \overset{\varphi'}{\longrightarrow} A_2 \longrightarrow 0,$$
where a matrix representation for $\psi'$ is $^t M$, a transposed matrix of $M$. Since elementary divisors of $^t M$ is equal to those of $M$, we have a form of $A_2$, 
$$A_2 = \mathbb{Z}^{n-l} \oplus \mathbb{Z} / d_1 \mathbb{Z} \oplus \cdots \oplus \mathbb{Z} / d_l \mathbb{Z}$$
and we have
$${\rm Tor}(A_2) = \mathbb{Z} / d_1 \mathbb{Z} \oplus \cdots \oplus \mathbb{Z} / d_l \mathbb{Z}.$$
Therefore ${\rm Tor}(A_1) \cong {\rm Tor}(A_2)$.
\end{proof}
\par
$A_1$ and $A_2$ have the same torsion part $\mathbb{Z} / d_1 \mathbb{Z} \oplus \cdots \oplus \mathbb{Z} / d_l \mathbb{Z}$, which have the information of tangle of $h_1$ and $h_2$, and this is what the linking numbers $Lk(h_1, h_2)$ indicate. We remark that if $d_i = 1$, $\mathbb{Z} / d_i \mathbb{Z}$ is a trivial group and does not appear explicitly in the torsion part. When we calculate the linking number from the torsion part, we also need the number $l$ which is the difference of rank between $A_1$ and $H_1(S^3\setminus h_1)$ or $A_2$ and $H_1(S^3\setminus h_2)$.

\ack{The author would like to thank Prof. Jun Murakami for helpful discussions. He also would like to thank Atsushi Ishii, Ayumu Inoue, Kengo Kishimoto, the participants of Tohoku Knot Seminar 2012 and the referee for their helpful comments.}

\end{document}